\documentclass{amsart}
\usepackage{amsfonts,amssymb,amscd,amsmath,enumerate,verbatim,calc}
\usepackage[all]{xy}

\newcommand{\wrt}{with respect to}

\newcommand{\zb}{\mathbf{z} }
\newcommand{\Z}{\mathbb{Z} }
\newcommand{\N}{\mathbb{N} }

\newcommand{\rt}{\rightarrow}

\newcommand{\image}{\operatorname{image}}

\newcommand{\grade}{\operatorname{grade}}

\newcommand{\charp}{\operatorname{char}}
\newcommand{\ann}{\operatorname{ann}}

\newcommand{\Sym}{\operatorname{Sym}}
\newcommand{\Hom}{\operatorname{Hom}}
\newcommand{\height}{\operatorname{height}}

\theoremstyle{plain}

\newtheorem{theorem}{Theorem}[section]

\newtheorem{lemma}[theorem]{Lemma}

\theoremstyle{definition}

\newtheorem{s}[theorem]{}
\newtheorem{remark}[theorem]{Remark}
\newtheorem{example}[theorem]{Example}

\theoremstyle{remark}

\begin{document}

\title[Fixed points]{Fixed points and grade of Hilbert polynomial of invariant rings}
\author{Tony~J.~Puthenpurakal}
\date{\today}
\address{Department of Mathematics, IIT Bombay, Powai, Mumbai 400 076, India}

\email{tputhen@gmail.com}

 \begin{abstract}
Let $k$ be a field and let $V$ be a $k$-vector space of dimension $d$. Let $G \subseteq GL(V)$ be a finite group. Let $r = \dim_k (V^*)^G$. Assume $r \geq 1$. Let $R = k[V]^G$ be the ring of invariants of $G$. Let $H_R(n) = a_{d-1}(n)n^{d-1} + \cdots a_1(n)n + a_0(n)$ be the Hilbert polynomial of $R$ where $a_i(-)$ are periodic functions. We show
 $a_{d-1}(-), \ldots, a_{d-r}(-)$ are constants. In the terminology of Erhart, $\grade H_R \leq d - r-1$. We also give an example which shows that our result is sharp.
\end{abstract}
 \maketitle
\section{introduction}
Quasi-polynomials are ubiquitous in math. Let $f \colon \N \rt \Z$ be a quasi-polynomial of degree $r$. Say $f(n) = a_r(n)n^r + \text{lower order terms}$. A natural question is when
$a_i(-)$ are constant for $i > \delta$. Following Erhart, the smallest such $\delta$ is called the \emph{grade} of $f$.
To put this question in the right historical framework let us consider the following:
Let $P$ be a $d$-dimensional rational convex polytope in $\mathbb{R}^n$ and let $E(P,k) = \sharp \ kP \cap \mathbb{Z}^n $ be the Ehrhart
function (here $\sharp V$ denotes cardinality of a set $V$). It is well-known that   $E(P,k)$  is of quasi-polynomial type cf. \cite[6.3.11]{BH}, say $E(P,n) = b_s(n)n^s + \text{lower terms}$.
 We state a
simplified version of a conjecture due to Ehrhart: \textit{If the affine span of every $d-1$ dimensional face of $P$ contains a point with integer co-ordinates then
$b_s(-)$ is a constant.}
A more general version of this was proved independently by  McMullen (see \cite{M}) and Stanley
(\cite{S}, Theorem 2.8). For a recent proof see \cite[5]{BI}.

The Hilbert function of a ring of invariants $R = k[V]^G$ provide a rich class of quasi-polynomials. In this paper we study Erhart's question in this case.

\s \label{setup} Let $k$ be a field and let $V$ be a $k$-vector space of dimension $d$. Let $G \subseteq GL(V)$ be a finite group. Let $V^* = \Hom_k(V, k)$ be the dual of $V$. Let $S = k[V] = \Sym(V^*)$, the symmetric algebra of $V^*$. Furthermore let $R = S^G$. Let $H_R$ be the Hilbert quasi-polynomial of $R$, i.e., $H_R(n) = \dim_k R_n$ for $n \gg 0$. It is known that $H_R(n) = a_{d-1}(n)n^{d-1} + \cdots a_1(n)n + a_0(n)$ is a quasi-polynomial and $a_{d-1}(-) \neq 0$. Let $r = \dim_k (V^*)^G$. Assume $r \geq 1$ (otherwise our result is vacuously true).

In this short paper we prove:
\begin{theorem}
\label{main}(with hypotheses as in \ref{setup}). Then  $a_{d-1}(-), \ldots, a_{d-r}(-)$ are constants.
\end{theorem}

\begin{remark}
It should be noted that we compute the dimension of fixed points of $V^*$ and \emph{not} of $V$. In fact the result as stated  is false if we work with fixed points of $V$; see \cite[8.2.6]{SL}.
\end{remark}

\begin{remark}
If $\charp k = p > 0$ and $G$ is a $p$-group then note that $\dim_k (V^*)^G \geq 1$, see \cite[8.2.1]{SL}.
\end{remark}
We now describe in brief the contents of this paper. In section two we prove a Lemma that we need. In section three we prove a result on quasi-polynomials that we need. In section four we prove Theorem \ref{main}. In the final section we give an example which shows that our result is sharp.
\section{A Lemma}
The following Lemma is a crucial ingredient in the proof of Theorem \ref{main}.
\begin{lemma}
\label{lem-dim}(with hypotheses as in \ref{setup}). Let $\{z_1, \ldots, z_t \}$ be a linearly independent subset of $(V^*)^G$. Then
\begin{enumerate}[\rm (1)]
  \item $\dim R/(z_1,\ldots, z_t) \leq d - t$.
  \item If $P$ is a minimal prime over $(z_1, \ldots, z_t)$ then $\height P = t$.
  \item There exists a homogeneous system of parameters $\{ z_1, \ldots, z_t, y_1, \ldots, y_{d-t}\}$ of $R$.
  \item The Hilbert series of $R$ is of the form
          $$ \frac{Q(u)}{(1-u)^t \prod_{j = 1}^{d-t}(1-u^{\deg y_j})} \quad \text{where} \ Q(u) \in \Z[u] \ \text{and} \ Q(1) \neq 0. $$
\end{enumerate}
\end{lemma}
\begin{proof}
(1) We prove the result by induction on $t$. When $t = 1$ we have nothing to prove. Assume the result for $t-1 \geq 1$ and we prove the result for $t$. The action of $G$ on $V^*/kz_1$ fixes $z_2, \ldots, z_t$. We have an exact sequence
$0 \rt S(-1) \xrightarrow{z_1} S \rt T \rt 0$. Taking invariants we have an exact sequence
$$0 \rt R(-1) \xrightarrow{z_1} R \xrightarrow{u} T^G \rt H^1(G, S)(-1) \xrightarrow{z_1} H^1(G, S), $$
where $H^1(G, S)$ is the first group cohomology of $S$ \wrt \ $G$. Set $A = \image u$  and $ K = \ker(H^1(G, S)(-1) \xrightarrow{z_1} H^1(G, S))$.
So we have an exact sequence $0 \rt A \rt T^G \rt K \rt 0$. We take Koszul homology \wrt \ $\zb^\prime = z_2, \ldots, z_t$. So we have an exact sequence
$$ H_1(\zb^\prime, K) \rt A/\zb^\prime A \rt T^G/\zb^\prime T^G \rt K/\zb^\prime K \rt 0.$$
By induction hypothesis we have $\dim T^G/\zb^\prime T^G \leq d-1 - (t-1) = d -t$. It follows that $\dim K/\zb^\prime K \leq d - t$. We note that $H_1(\zb^\prime, K)$ is annhilated by $\ann K$ and $\zb^\prime$.
So $\dim H_1(\zb^\prime, K) \leq d - t$. It follows that $\dim A/\zb^\prime A \leq d -t$. As $A = R/z_1 R$ we get that $\dim R/\zb R \leq d- t$; here $\zb = z_1, \ldots, z_t$.
The  result follows.

(2) Let $P$ be a minimal prime of $(z_1, \ldots, z_t)$. Then $P$ is a homogeneous prime. Then $\dim R/P \leq d - t$ by (1). We note that $R$ is equi-dimensional (as it is a domain) and catenary. So $\dim R/P = d - \height P$, see
\cite[Lemma 2, Section 31]{Mat} for the local case, the same proof works for the *-local case. So $\height P \geq t$. But $P$ is minimal over $(z_1, \ldots, z_t)$. So $\height P \leq t$. Thus $\height P = i$.

(3) By (2) it follows that $\height (z_1, \ldots, z_t) = t$. We can extend $\zb$ to get a homogeneous system of parameters of $R$.

(4) This follows from (3).
\end{proof}

\section{A result regarding quasi-polynomials}
In this section we prove a result regarding quasi-polynomials that we need.

\s Recall a function $f \colon \N \rt \Z$ is said to be of quasi-polynomial type if there exists a quasi-polynomial $P(X)$ such that $f(n) = P(n)$ for all $n \gg 0$. If
$P(n) = a_r(n)n^r + \cdots + a_0(n)$ with $a_i(-)$ periodic functions and $a_r(-) \neq 0$ then we say $\deg(f) = r$.

\s \label{qp} Let  $f \colon \N \rt \Z$ be of quasi-polynomial type of degree $r - 1$. Then
$$ \sum_{n \geq 0}f(n)z^n = \frac{Q(z)}{\prod_{i = 1}^{r} (1-z^{d_i})}   \quad \text{where} \ Q(z) \in \Z[z] \ \text{and} \ Q(1) \neq 0  \ \text{and} \ d_i \geq 1. $$

The following result is crucial in our proof of Theorem \ref{main}
\begin{theorem}
\label{qp-thm} Let  $f \colon \N \rt \N$ be of quasi-polynomial type of degree $d - 1$. Assume $$ \sum_{n \geq 0}f(n)z^n = \frac{Q(z)}{(1-z)^r\prod_{i = 1}^{d-r} (1-z^{d_i})} \quad \text{where} \ Q(z) \in \Z[z] \ \text{and} \ Q(1) \neq 0  \ \text{and} \ d_i \geq 1. $$
If  for $n \gg 0$ we have $f(n) = a_{d-1}(n)n^{d-1} + \cdots a_1(n)n + a_0(n)$  where $a_i(-)$ are periodic functions, then $a_{d-1}(-), \ldots, a_{d-r}(-)$ are constants.
\end{theorem}
\begin{proof}
We prove the result by induction on $r$. We first consider the case when $r = 1$.
We have
$$(1-z)\sum_{n \geq 0}f(n)z^n = \frac{Q(z)}{\prod_{i = 1}^{d-1} (1-z^{d_i})} = \sum_{n \geq 0}g(n).$$
Then $g$ is of polynomial type of degree $d - 2$. Say $g(n) = b_{d-2}n^{d-2} + \cdots$.
We note that $f(n) - f(n-1) = g(n)$. It follows that $a_{d-1}(-)$ is a constant as otherwise $\deg g $ will be $d -1$. The result follows when $r = 1$.

Now assume that result holds for $r -1$ and we  prove it for $r$ (here $r \geq 2$).
By induction hypothesis $a_{d-1}(-), \ldots, a_{d-r+ 1}(-)$ are constants.
We have
$$(1-z)\sum_{n \geq 0}f(n)z^n = \frac{Q(z)}{(1-z)^{r-1}\prod_{i = r}^{d-1} (1-z^{d_i})} = \sum_{n \geq 0}g(n).$$
Then $g$ is of polynomial type of degree $d - 2$. Say $g(n) = b_{d-2}n^{d-2} + \cdots$. Then again by induction hypotheses $b_{d-2}(-), \ldots, b_{d-r}(-)$ are constants. We note that $f(n) - f(n-1) = g(n)$. We get that $a_{d-r +1} + a_{d-r}(n) - a_{d-r}(n -1) = b_{d-r}$ which is a constant. As $a_{d-r}(-)$ is periodic it follows that $a_{d-r}(-)$ is a constant.
The result holds by induction.
\end{proof}
\section{Proof of Theorem \ref{main}}
We have
\begin{proof}[Proof of Theorem \ref{main}]
This follows from Lemma \ref{lem-dim} and Theorem \ref{qp-thm}.
\end{proof}
\section{An example}
We give an example which shows that our result is sharp.
\begin{example}
  Let $G = <\sigma> \subseteq GL_2(k)$ where
  $$   \sigma = \begin{pmatrix}
                  0 & 1 \\
                  1 & 0
                \end{pmatrix}                  $$
\end{example}
Then $k[V]^G = k[x+y, xy]$. We note that the dimension of fixed points of $V^*$ is one. The Hilbert series of this ring of invariants is
$H_R(n) = n/2 + a_0(n)$ for $n \gg 0$. Here $a_0(n) = 1$ if $n$ is even and $1/2$ when $n$ is odd. So $a_0(-)$ is \emph{not} a constant.

\end{document}